\documentclass{article}
\usepackage{graphicx}
\usepackage{amsthm}
\usepackage{amsmath}
\usepackage[inline]{asymptote}
\usepackage{tikz}
\usepackage[maxnames=10]{biblatex}
\renewbibmacro{in:}{}
\bibliography{main.bib}
\newtheorem{definition}{Definition}[section]
\newtheorem{theorem}{Theorem}[section]

\newtheorem{lemma}{Lemma}[section]

\newtheorem{question}{Question}[section]
\newtheorem{conjecture}{Conjecture}[section]
\newtheorem{claim}{Claim}[section]
\newtheorem*{claim*}{Claim}
\theoremstyle{remark}
\newtheorem{algorithm}{Algorithm}[section]
\newtheorem{construction}{Construction}[section]
\newtheorem*{remark*}{Remark}

\newcommand{\ep}{\varepsilon}
\newcommand{\R}{\mathcal R}
\newcommand{\LL}{\mathcal L}
\newcommand{\A}{\mathcal A}
\newcommand{\B}{\mathcal B}
\newcommand{\Ss}{\mathcal S}
\newcommand{\T}{\mathcal T}
\newcommand{\I}{\mathcal I}
\newcommand{\J}{\mathcal J}
\usepackage{xcolor}

\title{Spanning Structures in Multipartite Graph Traversals}
\author{Isabel McGuigan \footnote{Department of Mathematics, Carnegie Mellon University.  Email: iemcguigan@cmu.edu.}}
\date{\today}

\begin{document}

\maketitle

\abstract{Let $G$ be an $r$-partite graph such that the edge density between any two parts is at least $\alpha$.  We consider the problem of determining how large $\alpha$ must be in order to guarantee that $G$ has a Hamiltonian traversal (an $r$-cycle subgraph containing exactly one vertex from each part), and show that this critical density tends to $\frac 1 2$ as $r$ increases.  This resolves a conjecture of Badakhshian, Falgas-Ravry, and Sharifzadeh.  We also study the critical densities necessary to guarantee the existence of other spanning structures in traversals, particularly subgraph factors, and obtain asymptotically the critical densities for traversal $F$-factor subgraphs for several classes of graphs $F$.  The proofs of our results involve the absorption method.}

\section{Introduction}\label{sec:intro}

A central theme of extremal graph theory involves determining sufficient conditions for the existence of certain spanning structures in graphs.  Core results in this area include Dirac's theorem (and its many generalizations (\cite{berge}, \S 10.4)) giving sufficient conditions for the existence of a Hamilton cycle in a graph, and the theorem of Hajnal and Szemeredi \cite{hajnal} which gives a minimum-degree condition for the existence of a $K_t$-factor.  Multipartite versions of these theorems have also received attention; for example, Bondy \cite{bondy2} gives a degree condition for the existence of a Hamiltonian cycle in a balanced bipartite graph, and Lo and Markst\"om \cite{lo} and Keevash and Mycroft \cite{keevash} give a degree condition for the existence of a $K_t$-factor in a balanced multipartite graph.

In this paper we consider a different multipartite generalization of this class of problem, proposed in the following form by Badakhshian, Falgas-Ravry, and Sharifzadeh \cite{badakhshian}.  Let $r$ be an integer and $G$ be an $r$-partite graph with vertex partition $V = V_1\sqcup\dots\sqcup V_r$.  A \emph{traversal} of $G$ is a subgraph (on $r$ vertices)  containing exactly one vertex from each part.  For a family $\mathcal F$ of subgraphs of $K_r$, we say that $G$ contains an \emph{$\mathcal F$-traversal} if there exists a traversal of $G$ containing a member of $\mathcal F$ as a subgraph.  We are interested in density conditions for the existence of an $\mathcal F$-traversal in $G$ when $\mathcal F$ is some spanning structure (for example, an $r$-cycle).

In particular, for two disjoint vertex sets $X,Y\subset V[G]$, denote their edge density by $d(X,Y) := \frac{e(X,Y)}{|X||Y|}$.  The $r$-partite density of $G$ is the minimum edge density between two distinct parts, i.e.
\[d_r(G) := \min_{1\leq i<j\leq r}d(V_i,V_j).\]
For a family $\mathcal F$, we consider the critical density
\begin{align*}
    \pi_r(\mathcal F) := \sup \{\alpha \in [0, 1]: &\text{ there exists an $r$-partite graph $G$} \\&\text{with $d_r(G) \geq \alpha$ and no $\mathcal F$-traversal}\}.
\end{align*}

The consideration of density conditions of this form was first proposed by Bollob\'as (\cite{bollobas}, p. 324), and the problem of determining $\pi_r(\mathcal F)$ has been well-studied when $\mathcal F = \{F\}$ is a fixed graph.  The most general result in this direction is an analogue of the Erd\H os-Stone-Simonovitz Theorem:  Bondy, Shen, Thomass\'e, and Thomassen \cite{bondy} established that 
\[\lim_{r\to\infty}\pi_r(F)  =1 - \frac{1}{\chi(F)-1}\]
for all graphs $F$.  This result was strengthened by Pfender \cite{pfender} and subsequently by Narins and Tran \cite{narins} who showed
\[\pi_r(F) = 1-\frac{1}{\chi(F)-1}\]
for $r$ sufficiently large, if and only if $F$ satisfies a certain colorability condition.

The question of determining $\pi_r(\mathcal F)$ when $\mathcal F$ is a specific spanning structure was considered by Badakhshian, Falgas-Ravry, and Sharifzadeh in \cite{badakhshian}.  They in fact considered the problem in the more general setting of $H$-partite graphs, defined as follows.   Let $H$ be a graph on vertex set $V = \{v_1,\dots, v_r\}$.  An \emph{$H$-partite} graph $G$ is a subgraph of a blow-up of $H$; that is, an $r$-partite graph with vertex partition $V = V_1\sqcup \dots \sqcup V_r$ such that there are no edges between $V_i$ and $V_j$ unless $v_iv_j$ is an edge in $H$.  When $H = K_r$, an $H$-partite graph is simply an $r$-partite graph.  The $H$-partite density of an $H$-partite graph is given by
\[d_H(G) := \min_{\substack{1\leq i<j\leq r\\v_iv_j \in E(H)}}d(V_i,V_j).\]
And analogous to the definition of $\pi_r(\mathcal F)$, the critical density for a graph $H$ and a family $\mathcal F$ of subgraphs of $H$ is given by
\begin{align*}\pi_H(\mathcal F) := \sup \{\alpha \in [0, 1]: &\text{ there exists an $H$-partite graph $G$}\\&\text{ with $d_H(G) \geq \alpha$ and no $\mathcal F$-traversal}\}.\end{align*}

The study of $H$-partite graphs in this setting was initiated by Nagy \cite{nagy}, who studied the minimum density necessary to guarantee that an $H$-partite graph $G$ contains a copy of $H$ as a traversal subgraph; that is, the value of $\pi_H(H)$.

A natural question is to determine the smallest density forcing the existence of a connected traversal in an $H$-partite graph; that is, the value of $\pi_H(\mathcal T_r)$ when $H$ is a (connected) graph on $r$ vertices and $\mathcal T_r$ is the family of spanning trees on $r$ vertices.  Nagy's work determines this quantity when $H$ is a tree, and Badakhshian, Falgas-Ravry, and Sharifzadeh determined this quantity exactly for several small graphs $H$; they also conjectured values for $\pi_r(\mathcal T_r)$ and $\pi_{K_{r,r}}(\mathcal T_{2r})$.  The question for $\pi_r(\mathcal T_r)$ was resolved asymptotically by Lengler, Martinsson, Petrova, Schnider, Steiner, Weber, and Welzl in \cite{connectivity}, who showed \[\lim_{r\to\infty}\pi_r(\mathcal T_r) = \frac{3-\sqrt 5}{2} \approx 0.382.\]  For $\pi_{K_{r,r}}(\mathcal T_{2r})$, Badakhsian, Falgas-Ravry, and Sharifzadeh gave a simple construction illustrating $\pi_{K_{r,r}}(\mathcal T_{2r})\geq \frac 1 2$ for all $r$, and asked whether this bound might be tight.

\begin{question}[\cite{badakhshian}, question 5.6]\label{ques:bipconn}
Does $\pi_{K_{r,r}}(\mathcal T_{2r}) = \frac 1 2$ for all $r\geq 1$?
\end{question}

Badakhsian, Falgas-Ravry, and Sharifzadeh also considered the density threshold for a $C_r$, or ``Hamiltonian", traversal.  They gave a construction illustrating that $\pi_r(C_r)$ is strictly greater than $\frac 1 2$ for every $r \geq 3$; i.e., for every $r\geq 3$ there is an $r$-partite graph $G$ with $d_r(G)>\frac 1 2$ containing no Hamiltonian traversal.  However, they conjectured the asymptotic behavior of $\pi_r(C_r)$:

\begin{conjecture}[\cite{badakhshian}, conjecture 1.17]\label{conj:ham}
$\lim_{r\to\infty}\pi_r(C_r) = \frac 1 2$.
\end{conjecture}

The main result of this paper is a proof of Conjecture \ref{conj:ham}.

\begin{theorem}\label{thm:completeham}
For every $\ep > 0$, there exists $r_0$ such that, for all $r > r_0$, we have $\pi_r(C_r) < \frac 1 2 + \ep$.
\end{theorem}

In fact, we prove a slightly more specific result in the setting of $K_{r,r}$-partite graphs.

\begin{theorem}\label{thm:bipham}
For every $\ep > 0$, there exists $r_0$ such that, for all $r > r_0$, we have $\pi_{K_{r,r}}(C_{2r}) < \frac 1 2 + \ep$.
\end{theorem}

Clearly Theorem \ref{thm:bipham} implies Theorem \ref{thm:completeham} when $r$ is even, and a simple adjustment to the argument implies Theorem \ref{thm:completeham} when $r$ is odd.  The proof of Theorem \ref{thm:bipham} relies on the construction of a small absorber into which long paths in the rest of the graph may be connected to create a cycle, as well as a density lemma to enable the connections.

Theorem \ref{thm:bipham} immediately gives an asymptotic answer to Question \ref{ques:bipconn}, and using the same density lemma we obtain the following exact result.

\begin{theorem}\label{thm:bipcon}
For every $r > 7000$, we have $\pi_{K_{r,r}}(\mathcal T_{2r}) = \frac 1 2$.
\end{theorem}

No particular effort is made to optimize the constant 7000, though the method of proof fails for, say, $r=8$.  However, we can additionally show that $\pi_{K_{2,2}}(\mathcal T_4) = \pi_{K_{3,3}}(\mathcal T_6) = \frac 1 2$, and conjecture that indeed $\pi_{K_{r,r}}(\mathcal T_{2r}) = \frac 1 2$ for every $r$; see further discussion in Section \ref{sec:conclusion}.

Finally, we consider density conditions for the existence of traversal $F$-factor subgraphs, for various graphs $F$. In particular we consider the density threshold for a $K_t$-factor traversal; that is, the value of $\pi_{rt}(rK_t)$, where $rK_t$ denotes a collection of $r$ vertex-disjoint copies of $K_t$.  Badakhsian, Falgas-Ravry, and Sharifzadeh asked for this value; they observed that, by simply partitioning an $rt$-partite graph into $r$ $t$-partite graphs, we have $\pi_{rt}(rK_t) \leq \pi_t(K_t)$.  We give the following asymptotic result.

\begin{theorem}\label{thm:ktfactor}
For any $t\geq 3$, let $\alpha_t$ be the positive solution to the quadratic
\[\alpha = 1 - \frac{1}{t-2}\alpha^2;\]
that is, $\alpha_t = 1 - \frac{t-\sqrt{t^2-4}}{2}$.  Then $\pi_{rt}(rK_t) \geq \alpha_t$ for all $r$, and \[\lim_{r\to\infty}\pi_{rt}(rK_t) = \alpha_t.\]
\end{theorem}

Since Bondy, Shen, Thomass\'e, and Thomassen \cite{bondy} proved $\pi_3(K_3) = \frac{-1+\sqrt 5}{2} = \alpha_3$, when $t = 3$ this yields the exact result that $\pi_{3r}(rK_3) = \pi_3(K_3) = \frac{-1+\sqrt 5}{2}$ for all $r\geq 1$.  For $t > 3$ we have $\alpha_t < \pi_t(K_t)$, and we do not know whether or not $\pi_{rt}(rK_t) >\alpha_t$ for all $r$.

Similar to Theorem \ref{thm:bipham}, the proof of Theorem \ref{thm:ktfactor} involves the construction of an absorber.  It seems likely that similar techniques may be useful to determine asymptotically the critical density for a traversal $F$-factor for other graphs $F$; indeed with minor adjustments to the argument we also obtain the critical density for bipartite graph and odd cycle factors. 

\begin{theorem}\label{thm:bipfactor}
    Let $F$ be a bipartite graph on $t$ vertices which is not a star. Then $\lim_{r\to\infty} \pi_{rt}(rF) = \frac{3-\sqrt 5}{2}.$
\end{theorem}

\begin{remark*} The lower bound $\pi_{rt}(rF) \geq \frac{3-\sqrt 5}{2}$ holds for any $t$-vertex graph $F$ which is not a star.
\end{remark*}

\begin{theorem}\label{thm:oddcycfactor}
    Let $t\geq 7$ be odd.  Then
    $\lim_{r\to\infty} \pi_{rt}(rC_t) = \frac 1 2.$
\end{theorem}

We note that Theorems \ref{thm:ktfactor}, \ref{thm:bipfactor}, and \ref{thm:oddcycfactor} together give the values of $\lim_{r\to\infty}$ $\pi_{rt}(rC_t)$ for all $t\neq 5$; this gives an almost-complete asymptotic answer to another question of Badakhsian, Falgas-Ravry, and Sharifzadeh (\cite{badakhshian}, problem 5.20). 

The remainder of the paper is organized as follows.  In Section \ref{sec:prelim} we provide some definitions and notation, and state our main lemmas.  In Section \ref{sec:bipham} we prove Theorems \ref{thm:bipham} and \ref{thm:completeham}.  In Section \ref{sec:bipcon} we prove Theorem \ref{thm:bipcon}, and in Section \ref{sec:factors} we prove Theorems \ref{thm:ktfactor}, \ref{thm:bipfactor}, and \ref{thm:oddcycfactor}.  In Section \ref{sec:lemmaproof} we prove the main density lemma.  Finally in Section \ref{sec:conclusion} we discuss remaining open problems in this area, particularly the problem of determining $\pi_r(P_r)$, the density threshold for a Hamilton path traversal.

\section{Preliminaries}\label{sec:prelim}

We first fix some definitions and notation.  In all that follows, let $G$ be an $r$-partite graph with vertex partition $\mathcal P = \{P_1,\dots, P_r\}$.  For two disjoint vertex sets $P$ and $Q$, we denote by $N_P(Q)$ the set of vertices in $P$ which are adjacent to some vertex in $Q$.  We occasionally drop braces when convenient (for example, when $v$ is a vertex we write $N_P(v)$ for $N_P(\{v\})$, and $d(v,P)$ for $d(\{v\},P) = \frac{|N_P(v)|}{|P|}$).

We call a subgraph of $G$ which contains at most one vertex from each part a \emph{subtraversal}. We say that two subtraversal subgraphs of $G$ are \emph{part-disjoint} if there is no part of $G$ which contains a vertex of both of them.

In the remainder of this section we give some definitions and lemmas instrumental to the proof of Theorem \ref{thm:bipham}; we also briefly discuss weighted $H$-partite graphs.  The general strategy of the proof of Theorem \ref{thm:bipham} is to construct a small number of long paths within the $K_{r,r}$-partite graph, then connect them all to a small absorber, thereby creating a Hamiltonian cycle.  In order to ensure that the paths can be connected properly, we need to construct them with some amount of ``flexibility".  That is, we would like to specify the sequence of \emph{parts} that the vertices in a path come from, but have many choices for which actual \emph{vertices} within the parts we choose to form the path, particularly the endpoints.  This motivates the following definitions.

\begin{definition}
Let $P_1,P_2,\dots, P_k$, $Q$ be distinct vertex parts of $G$.  We say that $Q$ \emph{robustly connects} $P_1,\dots, P_k$ if there is a vertex $v\in Q$ such that $d(v,P_i)> \frac 1 2$ for all $1\leq i\leq k$.  We call $v$ the \emph{connecting vertex} in $Q$ for $P_1,\dots, P_k$, or just the ``connecting vertex" if $P_1,\dots, P_k$ are clear from context.
\end{definition}

\begin{definition}
We call a sequence of parts $P_1,Q_1,\dots,P_k,Q_k,P_{k+1}$ a \emph{robust sequence} (of length $2k$) if for each $1\leq i\leq k$, the part $Q_i$ robustly connects $P_i$ and $P_{i+1}$.  We consider a single part $P$ to be a trivial robust sequence (of length 0).  We refer to the parts $P_1$ and $P_{k+1}$ as the ``end parts" of the sequence.
\end{definition}

The utility of this definition is that a robust sequence in $\mathcal P$ must contain a path in $G$ (with many possible choices for the first and last vertex).  

\begin{lemma}\label{lem:robustpath}
    Let $P_1,Q_1,\dots,P_k,Q_k,P_{k+1}$ be a robust sequence.  Then there exist vertices $p_i\in P_i, q_j\in Q_j$ such that $p_1q_1\dots p_kq_kp_{k+1}$ is a path in $G$.
\end{lemma}

\begin{proof}
For each $1\leq i\leq k$, let $q_i$ be the connecting vertex in $Q_i$ for $P_i, P_{i+1}$.  For each $2\leq i\leq k$, since $d(q_{i-1}, P_i)$, $d(q_i,P_i) > \frac 1 2 $, we have $|N_{P_i}(q_{i-1})\cap N_{P_i}(q_i)| >  0$, so there's at least one vertex $p_i\in P_i$ adjacent to both $q_{i-1}$ and $q_i$.  By picking $p_1\in N_{P_1}(q_1)$ and $p_{k+1}\in N_{P_{k+1}}(q_k)$ arbitrarily, we complete the desired path.
\end{proof}

In general it is useful to think of a robust sequence as simply a path in $G$ whose end vertices are not fully specified.

Before beginning the proof we need two more lemmas.  The first enables the creation of long robust sequences, which contain long paths, as well as the absorber to which the long paths are ultimately connected.  It also underlies the proofs of Theorems \ref{thm:ktfactor}, \ref{thm:bipfactor}, and \ref{thm:oddcycfactor}.  The second allows for the construction of short paths connecting specific sets of vertices, which is what enables us to connect the long paths to the absorber; it is also useful in the proofs of Theorems \ref{thm:bipcon} and \ref{thm:oddcycfactor}.

\begin{lemma}\label{lem:robustconnect}
    Let $P_1,P_2,\dots, P_k$, $Q$ be distinct parts of $G$.  Suppose $\alpha\in [0,1]$ and $\ep > 0$ are such that $d(Q,P_i)>\alpha + \ep$ for all $1\leq i\leq k$.  Then there is a subset $S\subset[k]$ of size greater than $\ep k$ and a vertex $v\in Q$ such that $d(v, P_s) > \alpha$ for all $s\in S$. 
\end{lemma}

\begin{remark*}
    When $\alpha = \frac 1 2$, as in the setting of Theorem \ref{thm:bipham}, this implies that $Q$ robustly connects the parts $\{P_s:s\in S\}.$
\end{remark*}

\begin{proof}
The proof is a straightforward averaging argument.  From the density condition, we have $\sum_{i=1}^kd(Q, P_i) \geq \left(\alpha + \ep\right)k$.  Averaging over vertices in $Q$, we can write
\[\sum_{i=1}^kd(Q, P_i) = \frac 1 {|Q|}\sum_{v\in Q}\sum_{i=1}^kd(v,P_i),\]
so there must be some $v\in Q$ for which $\sum_{i=1}^kd(v,P_i) \geq  \left(\alpha + \ep\right)k$.  Let $m$ be the number of parts $P_i$ for which $d(v,P_i)>\alpha$.  If $m \leq \ep k$, we'd then have
\[\sum_{i=1}^kd(v,P_i) \leq (k-\ep k)\cdot \alpha + \ep k\cdot 1 < \left(\alpha+ \ep\right)k,\]
a contradiction.  Therefore there is a set $S\subset [k]$ of size greater than $\ep k$ such that $d(v, P_s) > \alpha$ for all $s\in S$.
\end{proof}

\begin{lemma}\label{lem:connect4strong}
    Let $P_1,P_2,P_3,P_4\in\mathcal P$ be four parts of $G$, and let $S_i\subset P_i$ be vertex subsets such that $\frac{|S_1|}{|P_1|} + \frac{|S_2|}{|P_2|}, \frac{|S_3|}{|P_3|} + \frac{|S_4|}{|P_4|}  \geq 1$.  Suppose that $d(P_i, P_j)>\frac 1 2$ for $i = 1,2$, $j = 3,4$.  Then, at least one of the following occurs:
    \begin{itemize}
        \item $G$ contains an edge between $S_1\cup S_2$ and $S_3\cup S_4$
        \item $G$ contains a path of length 3 with exactly one vertex in each $P_i$, whose first vertex is in $S_1\cup S_2$, second in $P_3\cup P_4$, third in $P_1\cup P_2$, and last in $S_3\cup S_4$.
    \end{itemize}
\end{lemma}

The proof of this lemma is by partitioning each $P_i$ into disjoint neighborhoods, bounding the densities $d(P_i, P_j)$ in terms of the sizes of these neighborhoods, and analyzing the resulting optimization problem.  We defer the calculations to Section \ref{sec:lemmaproof}. 

\subsection{Unweighted and weighted $H$-partite graphs}

All $r$-partite and $H$-partite graphs we have considered so far have been unweighted.  However, it is often useful to consider weighted $H$-partite graphs, defined as follows.  A weighted $H$-partite graph consists of an $H$-partite graph $G$ with vertex partition $V_1\sqcup\dots\sqcup V_r$ together with a weight function $w: V \to [0,1]$ satisfying $\sum_{v\in V_i}w(v) = 1$ for every part $V_i$.  The weighted density between two vertex sets $X$ and $Y$ is given by
\[d_w(X,Y) = \sum_{uv\in E(X,Y)}w(u)w(v),\]
and the weighted $H$-partite density is as before given by 
\[d_{w,H}(G) = \min_{\substack{1\leq i<j\leq r\\ v_iv_j\in E(H)}}d_w(V_i,V_j).\]

As observed by Nagy \cite{nagy}, the settings of weighted and unweighted graphs are equivalent in the following sense.  An unweighted $H$-partite graph can be viewed as a weighted $H$-partite graph with weight function $w(v) = \frac{1}{|V_i|}$ for each $v\in V_i$.  Conversely, a weighted $H$-partite graph can be arbitrarily well-approximated by unweighted $H$-partite graphs by taking rational approximations to its weight function, and considering appropriate blow-ups of the original weighted graph.  It follows that the value of $\pi_H(\mathcal F)$ is the same whether we consider the supremum to be taken over weighted or unweighted $H$-partite graphs.  For simplicity, we generally consider unweighted graphs in our upper-bound proofs and use weighted graphs for lower-bound constructions.

\section{Hamiltonian traversals in $K_{r,r}$-partite graphs}\label{sec:bipham}

\subsection{The lower bound}

Badakhsian, Falgas-Ravry, and Sharifzadeh showed in \cite{badakhshian} that $\pi_r(C_r) > \frac 1 2$ for all $r\geq 3$; that is, for all $r\geq 3$, there is an $r$-partite graph $G$ with $d_r(G)>\frac 1 2$ which has no Hamiltonian traversal. By monotonicity this implies $\pi_{K_{r,r}}(C_{2r}) > \frac 1 2$ for all $r\geq 2$ as well.  For completeness, we reproduce (a slightly simplified version of) their construction here.

\begin{construction}[\cite{badakhshian}, construction 4.3]\label{cons:lowerbound}
    We define a weighted $r$-partite graph $G$.  The parts of $G$ are $V_i = \{x_i, y_i\}$ for $i = 1, 2, \dots, r-2$, $V_{r-1} = \{x_{r-1}\}$, and $V_r = \{v_1,\dots, v_{r-1}\}.$  The edges are:
    \begin{itemize}
        \item $y_iy_j$ for all $1\leq i < j \leq r-2$, and $x_ix_j$ for all $1\leq i < j \leq r-1$
        \item $y_iv_j$ for all $1\leq i \leq r-2$ and $1\leq j\leq r-1$
        \item $x_iv_i$ for all $1\leq i \leq r-1$
    \end{itemize}

    See Figure \ref{fig:lowerbound}. This $G$ has no Hamiltonian traversal; any such traversal would have to contain each $x_i$ (as it must contain $x_{r-1}$, and the vertices chosen in $V_1,\dots, V_{r-1}$ must form a path), but then there is no choice of $v_i\in V_{r}$ which would complete the cycle.

    \begin{figure}
        \centering
        \includegraphics{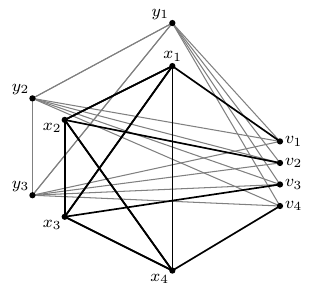}
        \caption{Construction \ref{cons:lowerbound} for $r=5$.}
        \label{fig:lowerbound}
    \end{figure}
    
    The weights are $w(x_i) = \frac 1 2 + \frac 1 {4r}$, $w(y_i) = \frac 1 2 - \frac 1 {4r}$ for $1\leq i\leq r-2$; $w(x_{r-1})$ = 1; $w(v_i) = \frac{1}{2r}$ for $1\leq i\leq r-2$; and $w(v_{r-1}) = 1-\frac{r-2}{2r}$. This yields the following densities:
    \begin{itemize}
        \item For $1\leq i < j \leq r-2$, $d_w(V_i,V_j) = \left(\frac 1 2 + \frac 1 {4r}\right)^2 + \left(\frac 1 2 - \frac 1 {4r}\right)^2 = \frac 1 2 + \frac{1}{8r^2}$. 
        \item For $1\leq i \leq r-2$, $d_w(V_i,V_{r-1}) = \frac 1 2 + \frac 1 {4r}$. 
        \item For $1\leq i \leq r-2$, $d_w(V_i,V_r) = \left(\frac 1 2 - \frac 1 {4r}\right) + \left(\frac 1 2 + \frac 1 {4r}\right)\left(\frac 1 {2r}\right) = \frac 1 2 + \frac 1 {8r^2}$.
        \item $d_w(V_{r-1},V_r) = 1 - \frac{r-2}{2r} = \frac 1 2 + \frac 1 r$.
     \end{itemize}

     So $G$ is an $r$-partite graph with no Hamiltonian traversal and $d_{w,r}(G) = \frac 1 2 + \frac 1 {8r^2} > \frac 1 2$.
    
\end{construction}

\subsection{Proof of theorem \ref{thm:bipham}}

Fix $\ep > 0$, let $r$ be sufficiently large, and let $G$ be a $K_{r,r}$-partite graph with $d_{K_{r,r}}(G) > \frac 1 2 + \ep$.  We show that $G$ has a Hamiltonian traversal.

A general outline of the proof is as follows.  Let $G$ have vertex partition $\mathcal P = \LL\sqcup \R = \{L_1,\dots, L_r\}\sqcup \{R_1,\dots, R_r\}$. First, we find small collections of parts $\A_L, \B_L\subset \LL$, $\A_R,\B_R\subset\R$ such that each $A\in \A_L$ robustly connects $\B_R$, and each $A\in \A_R$ robustly connects $\B_L$.  We refer to this structure as the ``absorber".  Second, we partition remaining parts into a small number of robust sequences.  Third, for each sufficiently long robust sequence we construct a long path which traverses most of the parts in the sequence and has both endpoints in the absorber.  Finally, we partition any leftover parts into robust sequences whose end parts are in the absorber, and continue lengthening these sequences until all parts are traversed by some path connected to the absorber, and the number of paths is exactly correct to complete a Hamiltonian traversal.  

\subsubsection{Constructing the absorber}\label{sec:absorber}

In this section we rigorously describe and construct the absorber. Fix constants 
\[C_1 = \frac {80} {\ep^2},\quad C_2 = \frac {150}{\ep^2}.\]  
\begin{claim}\label{claim:absorber}
    There exist disjoint sets $\A_L,\B_L\subset \LL$ and $\A_R,\B_R\subset \R$ satisfying the following:
    \begin{enumerate}
        \item $|\A_L| = |\A_R| = C_1$, $|\B_L| = |\B_R| = C_2$.
        \item Every $A\in\A_L$ robustly connects the parts in $\B_R$.
        \item Every $A\in\A_R$ robustly connects the parts in $\B_L$.
    \end{enumerate}
\end{claim}

\begin{proof}
    We first find $\A_L$ and $\B_R$.  For each part $L$ in $\LL$, by Lemma \ref{lem:robustconnect} there is a set $\Ss_L\subset \R$ of size at least $\ep r$ such that $L$ robustly connects $\Ss_L$.  Consider an auxiliary bipartite graph with left and right vertex sets $\LL$ and $\R$ respectively.  We connect $L\in \LL$ to $R\in \R$ if $R\in \Ss_L$.  The resulting bipartite graph has $2r$ vertices and at least $\ep r^2$ edges; by the Kővári–Sós–Turán theorem \cite{kst} it contains a copy of $K_{C_1,C_2}$ for $r$ sufficiently large.  Such a $K_{C_1,C_2}$ corresponds to sets $\A_L$ of size $C_1$ and $\B_R$ of size $C_2$ such that each $L\in\A_L$ robustly connects $\B_R$, as desired.

    We find $\A_R$ and $\B_L$ by the same argument.
\end{proof}

Claim \ref{claim:absorber} produces a structure with two sections: $\A_L\cup\B_R$ and $\A_R\cup\B_L$.  Each part in $\A_L$ (resp. $\A_R$) has a connecting vertex for $\B_R$ (resp. $\B_L$). We need to cross between these two sections, so we augment them with two extra paths. 

\begin{claim}\label{claim:shortpaths}
    There exist two part-disjoint subtraversal paths $P_1$ and $P_2$ in $G$ satisfying the following:

    \begin{enumerate}
        \item $P_1$ and $P_2$ each have one endpoint which is a connecting vertex in $\A_L$, and one which is a connecting vertex in $\A_R$.
            \item $P_1$ and $P_2$ each have length either 3 or 5, and their non-endpoint vertices are contained in $\B_L\cup \B_R$.
    \end{enumerate}

\end{claim}

\begin{proof}
We first find $P_1$.  Choose arbitrary parts $A_1,A_2\in \A_L$,  $A_3,A_4\in\A_R$, $B_1,B_2\in \B_R$, and $B_3,B_4\in \B_L$.   For each $1\leq i\leq 4$, let $a_i$ be the connecting vertex in $A_i$, and let $S_i = N_{B_i}(a_i)$.   As $a_i$ is the connecting vertex, $S_i$ satisfies $|S_i|>\frac 1 2 |B_i|$ for each $i$.  Thus, there exists a path from $S_1\cup S_2$ to $S_3\cup S_4$ as described by Lemma \ref{lem:connect4strong}, and hence one from $\{a_1,a_2\}$ to $\{a_3,a_4\}$ satisfying conditions (1) and (2).

By repeating this argument, choosing different arbitrary parts, we obtain the path $P_2$.
\end{proof}

The sets $\A_L,\A_R,\B_L,$ and $\B_R$, augmented with the paths $P_1$ and $P_2$, form our absorber.  In the sequel we denote by $\A_L'$ the set of parts in $\A_L$ which neither $P_1$ nor $P_2$ traverse, and use $\A_R',\B_L',\B_R'$ similarly.  We have $|\mathcal A_L'| = |\mathcal A_R'| = C_1 - 2$, and $C_2-4\leq |\mathcal B_L'| = |\mathcal B_R'| \leq C_2-2$.

\subsubsection{Constructing robust sequences}

Let $\LL' = \LL\setminus \left(\A_L\cup \B_L\right)$, $\R' = \R\setminus\left( \A_R\cup \B_R\right)$.  We have $|\LL'| = |\R'| = r - C_1-C_2$. In this section we discuss how to partition $\LL'\cup\R'$ into a small number of robust sequences.  

\begin{claim}\label{claim:robustpaths}
    There exists a partition of $\LL'\cup\R'$ into robust sequences satisfying the following:
    \begin{enumerate}
        \item Every nontrivial robust sequence has both its end parts in $\R'$.
        \item The total number of robust sequences is less than $\frac 2 \ep$.
    \end{enumerate}
\end{claim}

During this stage, and further in the argument, it is useful to keep track of the number of sequences we have created.  We use the following definition and notation.

\begin{definition}
    We call a sequence of parts $P_1,Q_1,\dots,P_k,Q_k,P_{k+1}$ \emph{alternating} if each $P_i$ is in $\LL$ and each $Q_i$ is in $\R$, or each $P_i$ is in $\R$ and each $Q_i$ is in $\LL$.
\end{definition}

\begin{remark*}
    Every sequence of parts which contains a path in $G$ (in particular, every robust sequence) is alternating.
\end{remark*}

\begin{definition}
    Given subsets $\Ss\subset\LL$ and $\T\subset\R$, and a partition of $\Ss\cup\T$ into (possibly trivial)   sequences of parts, let $N_{\Ss}$ be the number of (possibly trivial) sequences with both end parts in $\Ss$ and $N_{\T}$ be the number of (possibly trivial) sequences with both end parts in $\T$. 
\end{definition}

We have the following key observation.

\begin{lemma}\label{lem:nr}
Suppose that $\Ss\subset\LL$, $\T\subset\R$, and the parts of $\Ss\cup\T$ are partitioned into (possibly trivial) alternating sequences such that every sequence has both end parts in $\Ss$ or both end parts in $\T$.  If $|\Ss| = |\T|$, then $N_{\Ss} = N_{\T}$.
\end{lemma}

\begin{proof}
If the partition contains no nontrivial alternating sequences, we have $N_{\Ss} = |\Ss| = |\T| = N_{\T}$.  Consider the action of connecting some parts of $\Ss\cup\T$ into a nontrivial alternating sequence with both end parts in $\T$.  This sequence has length $2l$ for some $l\geq 1$, and uses $l+1$ parts from $\T$ and $l$ from $\Ss$.  Thus, this action decreases $N_{\Ss}$ by $l$, and decreases $N_{\T}$ by $l+1$ while increasing it by 1; hence after forming this sequence we still have $N_{\Ss} = N_{\T}$.  Iterating proves the claim.
\end{proof}

\begin{proof}[Proof of Claim \ref{claim:robustpaths}]

We give a simple algorithm to produce the desired partition, which is also useful in later stages of the proof of Theorem \ref{thm:bipham}.

\begin{algorithm}\label{alg:robustpaths}
    Given two subsets $\Ss\subset\LL$ and $\T\subset\R$ (or $\Ss\subset\R$ and $\T\subset\LL$), we repeat the following steps.  We begin with the trivial partition obtained by taking each part in $\Ss\cup \T$ to be a trivial robust sequence.  At each step, we consider a part in $\mathcal P$ to be ``isolated" if it is not yet part of some nontrivial robust sequence.

    \begin{enumerate}
\item Define a set $\mathcal C\subset\T$ of ``connectable" parts consisting of one end part from each nontrivial robust sequence.  If $|\mathcal C| < \lceil\frac 1 \ep\rceil$, then add isolated parts to it until $|\mathcal C| = \lceil\frac 1 \ep\rceil$.
If there are not enough isolated parts to do this, halt. Otherwise continue to step 2.
\item If there are no isolated parts in $\Ss$, halt.  Otherwise, choose any isolated part $S\in \Ss$.  Since $\ep |\mathcal C| \geq 1$, by Lemma \ref{lem:robustconnect}, there are at least two parts $T_1, T_2$ in $\mathcal C$ that $S$ robustly connects.  Each of $T_1$ and $T_2$ was the end part of a (possibly trivial) robust sequence; by inserting $S$ between them, we connect those two sequences together.  Return to step 1.
\end{enumerate}
\end{algorithm}

The result of this algorithm is a partition of $\Ss\cup\T$ into robust sequences such that each nontrivial sequence has both its end parts in $\T$.

\begin{lemma}\label{lem:alganal}
The following statements are true of Algorithm \ref{alg:robustpaths}:
\begin{enumerate}
    \item At no point in this process are there more than $\lceil\frac 1 \ep\rceil$ disjoint nontrivial robust sequences in $\Ss\cup\T$.
    \item Each iteration of Algorithm \ref{alg:robustpaths} decreases $N_{\Ss}$ and $N_{\T}$ by exactly 1.
    \item At the point when the algorithm halts, either $N_{\Ss} = 0$ or $N_{\T} < \frac 1 \ep$.  If $|\Ss| = |\T|$ to begin with then at the conclusion of the algorithm we have $N_{\Ss} = N_{\T} < \frac 1 \ep$.
\end{enumerate}
\end{lemma}

\begin{proof}
Note first that each iteration of step 2 changes the number of disjoint nontrivial robust sequences in $\Ss\cup\T$ by at most 1.  

To see (1), suppose that at at some point in this process there were more than $\lceil\frac 1 \ep\rceil$ disjoint nontrivial robust sequences in $\LL'\cup\R'$. Let iteration $i$ be the first at which this occurs.  Then after step 2 of iteration $i-1$, the number of disjoint nontrivial robust sequences must have been exactly equal to $\lceil\frac 1 \ep\rceil$.  But then at step 1 of iteration $i$, $\mathcal C$ would consist of exactly one end part from each of these $\lceil\frac 1 \ep\rceil$ nontrivial sequences, and step 2 would connect two of these end parts together, decreasing the number of disjoint sequences rather than increasing it. 

(2) is clear.

To see (3), note that the algorithm ends either when $|\mathcal C| = \min(\lceil\frac 1 \ep\rceil, N_{\T})$ is less than $ \frac 1 \ep$, or when there is no remaining isolated part in $\Ss$; since there are no nontrivial robust sequences with both endpoints in $\Ss$, this means $N_{\Ss} = 0$.  If $|\Ss|=|\T|$ then by Lemma \ref{lem:nr} we have $N_{\Ss} = N_{\T}$ at every step of the process, implying the second statement.
\end{proof}

By applying Algorithm \ref{alg:robustpaths} to $\Ss= \LL'$ and $\T = \R'$, we obtain a partition of $\LL'\cup\R'$ into robust sequences such that each such nontrivial sequence has both endpoints in $\R'$ and $N_{\LL'} = N_{\R'} < \frac 1 \ep$, proving Claim \ref{claim:robustpaths}.
\end{proof}

\subsubsection{Constructing long paths connected to the absorber}

Claim \ref{claim:robustpaths} provides a partition of $\LL'\cup\R'$ into a small number of robust sequences.  The next step is to find short paths which connect the ends of sufficiently long robust sequences to the absorber.  The result is a collection of ``long paths" traversing most of the parts in $G$ whose endpoints are connecting vertices in $\A_R'$.

\begin{claim}\label{claim:connect}
    There exists a collection of subtraversal paths $Q_1,Q_2,\dots$ in $G$ satisfying the following:
    \begin{enumerate}
        \item The $Q_i$'s are part-disjoint, and are part-disjoint from the paths $P_1$ and $P_2$ produced by Claim \ref{claim:shortpaths}.
        \item The endpoints of each $Q_i$ are connecting vertices in $\A_R'$.
        \item Each $Q_i$ traverses at most 4 parts in $\B_L'$, and no parts in $\B_R$ or $\A_L$.
        \item The total number of paths $Q_i$ is at most $\frac 2 \ep$.
        \item The paths $Q_1,Q_2,\dots$ together traverse all but at most $\frac {66}{\ep^2}$ parts in $\LL'\cup\R'$.
    \end{enumerate}
\end{claim}

\begin{proof} The following lemma enables connecting both ends of a sufficiently long robust sequence to parts in $\B_L'$. 

\begin{lemma}\label{lem:absorberconnect}
Let $C = \left\lceil\frac 4 \ep\right\rceil$.  Let $R_1,L_1,\dots,R_k,L_k,R_{k+1}$ be a robust sequence of length $2k > 4C$ with each $R_i\in\R$, $L_i\in\LL$. Let $B_1, B_2, B_3, B_4$ be four other parts in $\LL$ with subsets $S_i\subset B_i$ satisfying $|S_i|>\frac 1 2|B_i|$, and let $T_1, T_2$ be two more arbitrary parts in $\LL$.  Then there is a subtraversal path in $G$ which has one endpoint in $S_1 \cup S_2$, the other in $S_3\cup S_4$, and traverses all but at most $4C$ of the parts $R_1,L_1,\dots,R_k,L_k,R_{k+1}$.
\end{lemma}  

\begin{proof}
We first connect the beginning of the given robust sequence to $S_1\cup S_2$.  Since $\ep C\geq 4$, by Lemma \ref{lem:robustconnect} there is a set $S\subset [C]$ of size at least 4 such that $T_1$ robustly connects $\{R_s:\;s\in S\}$.  In particular, there are some $1\leq i'< j'\leq C$ such that $j' - i' \geq 3$ and $T_1$ robustly connects $R_{i'}$ and $R_{j'}$.  Choose $i, j$ with $i' < i < j < j'$. 

Let $l_{i-1}, l_j$ be the connecting vertices in $L_{i-1}, L_j$.  Let $S_{i}  = N_{R_{i}}(l_{i-1})$ and $S_{j} = N_{R_{j}}(l_j)$, so $|S_i|>\frac 1 2 |R_i|$ and $|S_j|>\frac 1 2 |R_j|$.  By Lemma \ref{lem:connect4strong}, there exists a path from $S_1\cup S_2$ to $S_i\cup S_j$ using at most one vertex in each of $B_1, B_2, R_i, R_j$.  Without loss of generality suppose this path has one endpoint in $S_1$.  If the other endpoint is in $S_j$, then the sequence
\[S_1, \dots, S_j, L_j, R_{j+1},\dots, R_k,L_k,R_{k+1}\]
contains a path in $G$ (where the first ellipses represent either zero or two other parts traversed by the path from $S_1$ to $S_j$).  If the other endpoint is in $S_i$, then the sequence 
\[S_1, \dots, S_i, L_{i-1},R_{i-1},\dots,L_{i'}, R_{i'}, T_1, R_{j'}, L_{j'},\dots, R_k,L_k,R_{k+1}\]
likewise contains a path in $G$. See Figure \ref{fig:absorberconnect}.

\begin{figure}
\centering
\includegraphics[width=\textwidth]{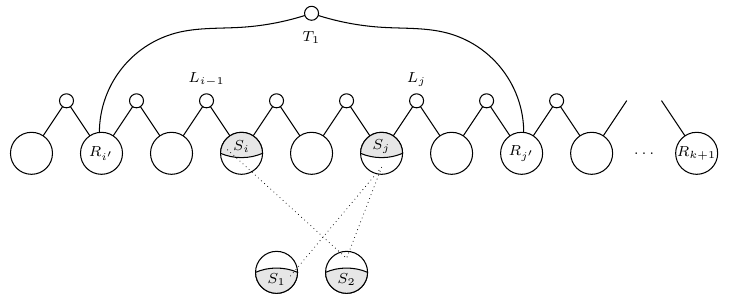}
\caption{Construction of the path in Lemma \ref{lem:absorberconnect}.  In this case there is a path from $S_1$ to $S_i$ using at most one vertex in each of $B_1, B_2, R_i, R_j$, and the sequence $S_1,R_j,B_2,S_i,L_{i-1},R_{i-1},\dots,L_{i'}, R_{i'},T_1,R_{j'},L_{j'},\dots, R_k,L_k,R_{k+1}$ contains a path in $G$.}
\label{fig:absorberconnect}
\end{figure}

In both cases the path produced traverses every part in $R_1,L_1,\dots,R_k,$ $L_k,R_{k+1}$ except for some subset of $R_1,L_1,\dots, R_C,L_C$, so misses at most $2C$ parts.

By repeating this argument on the other end of the original sequence, using $B_3, B_4, T_2$ in place of $B_1, B_2, T_1$, we obtain a path in $G$ which has one endpoint in $S_1 \cup S_2$, the other in $S_3\cup S_4$, and traverses all but at most $4C$ of the parts in the original sequence.
\end{proof}

For each robust sequence $R_1,L_1,\dots,R_k,L_k,R_{k+1}$ of length greater than $4C$ in our partition of $\LL'\cup\R'$, we may choose different arbitrary parts $A_1,\dots, A_4\in\A_R'$, $B_1,\dots,B_4,$ $T_1,T_2\in \B_L'$.  For each $1\leq i\leq 4$, the part $A_i$ robustly connects all of $\B_L$; let $a_i\in A_i$ be the connecting vertex, and let $S_i = N_{B_i}(a_i)$ for each $1\leq i \leq 4$.  Applying Lemma \ref{lem:absorberconnect} yields a path in $G$ from $S_1\cup S_2$ to $S_3\cup S_4$ (and hence one from $\{a_1,a_2\}$ to $\{a_3,a_4\}$) which traverses all but at most $4C$ parts of $R_1,L_1,\dots,R_k,L_k,R_{k+1}$.  The path thus produced is one of the desired $Q_i$'s.

Because we started with at most $\frac 2 \ep$ nontrivial robust sequences in $\LL'\cup\R'$, and $|\A_R'|  = C_1-2 > \frac 8 \ep$ and $|\B_L'| \geq C_2 -4 > \frac{12} \ep$, we never run out of parts in $\A_R$ or $\B_L$ with which to perform these operations.  The construction of the $Q_i$'s, including the choice of different arbitrary parts, guarantee that the resulting paths satisfy conditions (1)-(4) of Claim \ref{claim:connect}.

For condition (5), we note that any part in $\LL'\cup\R'$ which is not contained in any $Q_i$ must have been one of the following:
\begin{itemize}
    \item An element of a robust sequence of length less than $4C$, which was too short to apply Lemma \ref{lem:absorberconnect} to.  Since we started with at most $\frac 2 \ep$ robust sequences, there are at most $\frac 2 \ep \cdot (4C+1) < \frac{34}{\ep^2}$ such parts.
    
    \item A part in a long robust sequence which was not traversed by the corresponding path created by Lemma \ref{lem:absorberconnect}, of which there are again at most $\frac 2 \ep \cdot 4C < \frac{32}{\ep^2}$.
\end{itemize}
So the total number of such parts is at most $\frac {66}{\ep^2}$ as desired.

\end{proof}

\subsubsection{Absorbing leftover parts}

At this point, we have a collection of long paths $Q_i$ satisfying the conditions of Claim \ref{claim:connect}, along with at most $\frac{66}{\ep^2}$ ``leftover" parts in $\LL'\cup\R'$ that the $Q_i$'s don't traverse.  We now apply Algorithm $\ref{alg:robustpaths}$ two more times in order to connect the leftover parts into the absorber and finish the cycle.  

For any set $\Ss\subset\mathcal P$ of parts, let $\I_{\Ss}$ be the set of parts in $\Ss$ which are not traversed by any of the paths $Q_i$.

\begin{claim}\label{claim:absorb1}
    There are disjoint nontrivial robust sequences $T_1, T_2,\dots$ of parts in $\I_{\R'}\cup\I_{\B_L'}$ satisfying the following conditions:
    
    \begin{enumerate}
        \item Each $T_i$ has both end parts in $\I_{\B_L'}$.
        \item Every part in $\I_{\R'}$ is contained in some $T_i$.
        \item The number of parts in $\I_{\B_L'}$ which are not contained in any of the $T_i$'s is at least $\frac {71}{\ep^2}$.
        
    \end{enumerate}
\end{claim}

\begin{proof}
    Apply Algorithm \ref{alg:robustpaths} with $\Ss= \I_{\R'}$ and $\T = \I_{\B_L'}$, and take $T_1,T_2,\dots$ to be the resulting \emph{nontrivial} robust sequences. This automatically satisfies condition (1).  Before beginning the algorithm, by conditions (3)-(5) of Claim \ref{claim:connect}, we have $N_{\I_{\R'}} = |\I_{\R'}| \leq \frac{66}{\ep^2}$ and \[N_{\I_{\B_L'}} = |\I(\B_L')| \geq |\B_L'| -4\cdot \frac 2 \ep \geq C_2 - 4 - \frac {8} \ep \geq \frac{138}{\ep^2}.\] In particular $N_{\I_{\B_L'}} > N_{\I_{\R'}} + \frac 1 \ep$ before beginning the algorithm.  Hence by Lemma \ref{lem:alganal} (2) and (3), at the conclusion of the algorithm, we have $N_{\I_{\R'}} = 0$ and $N_{\I_{\B_L'}} \geq \frac{138}{\ep^2} - |\I_{\R'}| \geq \frac{72}{\ep^2}$.  The first equality implies that there are no more isolated parts in $\I_{\R'}$, so condition (2) is satisfied.  And Lemma \ref{lem:alganal} (1) guarantees that at most $\frac 1 \ep$ of the paths counted by $N_{\I_{\B_L'}}$ are nontrivial, so condition (3) is also satisfied.
\end{proof}

Let now $\J_{\B_L'}\subset \I_{\B_L'}$ be the set of parts in $\B_L'$ which are not traversed by any of the paths $Q_i$ or contained in any $T_i$. Condition (3) of Claim \ref{claim:absorb1} implies $|\J_{\B_L'}| \geq \frac{71}{\ep^2}$.

\begin{claim}\label{claim:absorb2}
There are disjoint (possibly trivial) robust sequences $U_1,U_2,$ $\dots,  U_{C_1-1}$ of parts in $\left(\I_{\LL'}\cup \J_{\B_L'}\right) \cup \B_{R}'$ satisfying the following:
\begin{enumerate}
        \item Each $U_i$ has both end parts in $\B_R'$.    \item Every part in $\B_{R}'$ is contained in a $U_i$.
        \item Every part in $\I_{\LL'}$ is contained in a nontrivial $U_i$.

    \end{enumerate}
    
\end{claim}

\begin{proof}
    
Apply Algorithm \ref{alg:robustpaths} once more with $\Ss= \I_{\LL'}\cup \J_{\B_L'}$ and $\T = \B_{R}'$, with the following modifications:
\begin{itemize}
    \item In step 2, always take the isolated part from $\I_\LL'$ if possible; in other words only take it from $\J_{\B_L'}$ once there are no longer any isolated parts in $\I_\LL'$.
    \item At any step, if $N_{\B_R'} = C_1 - 1$, halt the algorithm.
\end{itemize}

Take the $U_i$'s to be all resulting (possibly trivial) robust sequences containing any part in $\B_R'$.  This guarantees conditions (1) and (2) are satisfied.  

Recall that by Lemma \ref{lem:alganal} each iteration of Algorithm \ref{alg:robustpaths} decreases $N_{\Ss}$ and $N_\T$ by 1.  In its original form the algorithm halts when either $N_{\Ss} = 0$ or $N_\T  < \frac 1 \ep$. With the modifications we halt when $N_{\Ss} = 0$ or $N_\T  = C_1-1$. Before beginning the algorithm, we have
\[C_2\geq |\B_R'| = N_{\B_R'} \geq C_2 - 4.\]
It follows that it takes at least $C_2 - 4 - (C_1 - 1)$, and at most $C_2 - (C_1 - 1)$, iterations of the algorithm for $N_{\B_R'}$ to equal $C_1 - 1$.  Since $|\J_{\B_L'}| \geq \frac{71}{\ep^2} \geq \frac{70}{\ep^2} + 1 = C_2 - (C_1-1) $, we do not run out of isolated parts in $\I_{\LL'}\cup \J_{\B_L'}$ with which to perform step 2 of the algorithm. Hence the algorithm really does halt when $N_{\B_R'} = C_1-1$, rather than stopping earlier; this guarantees that the number of $U_i$'s is precisely $C_1-1$. But since $|\I_{\LL'}| < \frac{66}{\ep^2}< \frac{70}{\ep^2}-3 = C_2 - 4  -( C_1 - 1)$, all of the parts in $\I_{\LL'}$ get used up.  This guarantees that condition (3) is satisfied.  

\end{proof}

\subsubsection{Finishing the cycle}

At this point, we have created the following paths and robust sequences in $G$:
\begin{itemize}
    \item The two short paths $P_1$ and $P_2$ given by Claim \ref{claim:shortpaths}
    \item The long paths $Q_i$ given by Claim \ref{claim:connect}
    \item The robust sequences $T_i$ given by Claim \ref{claim:absorb1}
    \item The robust sequences $U_i$ given by Claim \ref{claim:absorb2}.
\end{itemize}

The endpoints of every path $Q_i$ are connecting vertices in $\A_R'$.   Let $\mathcal Q'$ be the set of parts in $\mathcal A_R'$ which are not traversed by any path $Q_i$.  Number the elements of $\{Q_1,Q_2,\dots\}\cup \mathcal Q'$ by $Q_1',Q_2',\dots, Q_q'$.

The end parts of every sequence $T_i$ are in $\B_L'$.  Let $\mathcal T'$ be the set of parts in $\B_L'$ which are not traversed by any $Q_i$ nor contained in any $T_i$ or $U_i$.  Number the elements of $\{T_1,T_2,\dots\}\cup\mathcal T'$ by $T_1',T_2',\dots, T_t'$.

None of the parts in $\A_L'$ are traversed by any $Q_i$ or contained in any $T_i$ or $U_i$; number the parts of $\A_L'$ by $A_1,A_2,\dots, A_{C_1-2}$.

All of the parts in $\B_R'$ are traversed by one of $U_1, U_2,\dots, U_{C_1-1}$.

The $P_i$'s, $Q_i'$'s, $T_i'$'s, $A_i$'s, and $U_i$'s together are a collection of some subtraversal paths and some robust sequences such that every part in $G$ is traversed by some path or contained in some sequence.  There are $C_1-2$ of the $A_i$'s and $C_1-1$ of the $U_i$'s.  The final step is to compare the number of $Q_i'$'s and $T_i'$'s.

\begin{claim}
    $q + 1 = t$.
\end{claim}
\begin{proof}

By replacing every subtraversal path among the $P_i$'s, $Q_i'$'s, $T_i'$'s, $A_i$'s, and $U_i$'s with the sequence of parts it traverses, we obtain a partition of $\LL\cup \R$ into alternating sequences. The sequences induced by $P_1$ and $P_2$ contain the same number of parts in $\LL$ as in $\R$, and all the remaining sequences have both end parts in $\LL$ or both in $\R$.  Thus, the argument of Lemma \ref{lem:nr} applies, so for this partition $N_{\LL} = N_{\R}$.  In addition, we have
\[N_{\LL} = C_1 - 2 + t, \quad N_{\R} =  q + C_1 - 1\]
as the sequences with both end parts in $\LL$ are the $T_i'$'s and $A_i$'s, and the sequences with both parts in $\R$ are the $Q_i'$'s and $U_i$'s.  So $t$ must equal $q+1$ as desired.

\end{proof}
Since $q + 1 = t$, the sequence $T_1'Q_1'T_2'\dots Q_{q}'T_t'$ contains every $T_i'$ and $Q_i'$. The following sequence of paths and robust sequences now traverses or contains every part in $\LL\cup \R$:
\[T_1'Q_1'T_2'\dots Q_{q}'T_t'P_1U_1A_1U_2A_2\dots U_{C_1-2}A_{C_1-2}U_{C_1-1}P_2.\]
Since each $T_i', U_i$ is a robust sequence with both end parts in $\B_L$ or $\B_R$, and each $P_i, Q_i', A_i$ is either a single part containing a connecting vertex or has endpoints which are connecting vertices in $\A_L$ or $\A_R$, this sequence contains a Hamiltonian cycle in $G$, completing the proof.

\subsection{Proof of Theorem \ref{thm:completeham}}
Fix $\ep > 0$, let $r$ be sufficiently large, and let $G$ be an $r$-partite graph with $d_r(G) > \frac 1 2 + \ep$.  If $r$ is even, Theorem \ref{thm:bipham} immediately implies that $G$ contains a Hamiltonian traversal.  

If $r = 2r'+1$ is odd, then partition the parts of $G$ into  sets $\LL$ of size $r'-1$ and $\R$ of size $r'+2$. We apply the argument of Theorem \ref{thm:bipham}, except when constructing the absorber, we take $\A_L$ to have size $C_1+1$ instead of $C_1$, and $\B_R$ to have size $C_2 + 4$ instead of $C_2$; we now have $|\LL'| = |\R'| = r'-C_1-C_2-2$.  

Pick arbitrary parts $A_1,\dots, A_4 \in \A_L$, $B_1,\dots, B_4\in\B_R$, and for each $1\leq i\leq 4$ let $a_i$ be the connecting vertex in $A_i$ and $S_i = N_{a_i}(B_i)$.  By Lemma \ref{lem:connect4strong} there is a path from $S_1\cup S_2$ to $S_3\cup S_4$ of length 1 or 3, and hence a path $P_1$ from $\{a_1,a_2\}$ to $\{a_3,a_4\}$ of length 3 or 5.  The partition of $\mathcal P$ into sequences induced by $P_1$ (where every part not traversed by $P_1$ is a trivial sequence) has exactly one nontrivial sequence, which has both end parts in $\A_L$, so we have $N_{\A_L} = N_{\A_R} = C_1$ and $N_{\B_R} = C_2$ or $C_2 + 2$.  In the latter case we may distribute one of its isolated parts to $\LL'$ and one to $\R'$ so that we have $N_{\B_R} = N_{\B_L} = C_2$ and $|\LL'| = |\R'|$.  At this point we may apply the remainder of the argument with no further changes to find a Hamiltonian cycle in $G$.

\section{Connected traversals in $K_{r,r}$-partite \\ graphs}\label{sec:bipcon}

In this section we prove Theorem \ref{thm:bipcon}, that $\pi_{K_{r,r}}(\mathcal T_{2r}) = \frac 1 2$ for sufficiently large $r$.  The lower bound $\pi_{K_{r,r}}(\mathcal T_{2r}) \geq \frac 1 2$ is given by a construction of Badakhsian, Falgas-Ravry, and Sharifzadeh, which we reproduce here for completeness.  

\begin{construction}[\cite{badakhshian}, construction 3.4]
Let $H$ be a graph on vertex set $v_1,\dots, v_r$ which is not complete.  We may suppose without loss of generality that $v_1v_2$ is not an edge in $H$.  We now define the following unweighted $H$-partite graph $G$.  The parts of $G$ are $V_1 = \{x_1\}$, $V_2 = \{y_2\}$, and $V_i = \{x_i,y_i\}$ for $2<i\leq r$.  There are edges $x_ix_j$ for all $i\neq j \in [r]\setminus \{2\}$ and $y_iy_j$ for all $i\neq j \in [r]\setminus \{1\}$.  We now have $d(V_i,V_j) = \frac 1 2$ for all $\{i,j\}\neq\{1,2\}$; thus $d_H(G) = \frac 1 2$, but $G$ has no connected traversal as $x_1$ and $y_2$ are in different connected components.
\end{construction}

\begin{proof}[\unskip\nopunct]We now prove the upper bound.  Let $G$ be a $K_{r,r}$-partite graph with $r > 7000$ and $d_r(G) > \frac 1 2$. Let $\mathcal L, \mathcal R$ be, respectively, the sets of parts on the left and right sides of $G$.  A general outline of the proof is as follows.  First, we find a small structure to which the remainder of the graph may be connected.  Then, we partition the remainder of the graph into a small number of connected components, and finally use Lemma \ref{lem:connect4strong} to connect all the components together.

The first step is accomplished by the following claim.

\begin{claim}\label{claim:bipabsorb}
Let $C = 2\lceil\log_2 r\rceil$.  Then $G$ contains distinct parts $A_1,\dots, A_C$ $ \in \mathcal L$, $B_1,\dots, B_{2C+1} \in \mathcal R$ and vertices $v_i\in A_i$ and $w \in B_{2C+1}$ satisfying the following:
\begin{enumerate}
\item $d(v_i,B_{2i-1}) + d(v_i,B_{2i}) \geq 1$ for each $1\leq i\leq C$.
\item $w$ is adjacent to each $v_i$.
\end{enumerate}

\end{claim}

\begin{proof}

We iterate the following process to produce a sequence of parts $A_1,\dots,$ $A_{25C} \in \mathcal L$, $B_1,\dots, B_{50C} \in \mathcal R$ with vertices $v_i\in A_i$ such that $d(v_i,B_{2i-1}) + d(v_i,B_{2i}) \geq 1$ for each $i$.  On step $i$:
\begin{enumerate}
\item Choose $A_i$ to be any part in $\mathcal L\setminus \{A_1,\dots, A_{i-1}\}$.
\item Let $\mathcal B_i = \mathcal R\setminus \{B_1,\dots, B_{2(i-1)}\}$.  By averaging (as in the proof of Lemma \ref{lem:robustconnect}), we may choose a vertex $v_i\in A_i$ such that 

$\sum_{B\in \mathcal B_i}d(v_i,B) > \frac 1 2 |\mathcal B_i|$.
\item Again by averaging, we may choose $B_{2i-1}, B_{2i}\in \mathcal B_i$ such that 

$d(v_i,B_{2i-1})+d(v_i,B_{2i}) \geq 1$.
\end{enumerate}
Let $\mathcal R' = \mathcal R\setminus \{B_1,\dots, B_{50C}\}$, with size $r-50C = r-100\lceil \log_2 r\rceil > 0.8r$ (for $r > 7000$).  For each $1\leq i\leq 25C$, $v_i$ was chosen such that
$\sum_{B\in \mathcal B_i}d(v_i,B) > \frac 1 2 |\mathcal B_i|.$
Let $m$ be the number of $B\in \mathcal B_i$ with $d(v_i,B)<0.2$.  Then, we have
\[\frac 1 2 |\mathcal B_i|< \sum_{B\in \mathcal B_i}d(v_i,B) < 1\left(|\mathcal B_i|-m\right) + 0.2m,\]
so $m < \frac 5 8 |\mathcal B_i|  < \frac 5 8 r < \frac{25}{32}|\mathcal R'| < 0.8 |\mathcal R'|$.

It follows that for each $i$, at least $0.2|\mathcal R'|$ parts $B\in \mathcal R'$ have $d(v_i,B) > 0.2$.  Thus, there is some part $\tilde B\in \mathcal R'$ such that $d(v_i,\tilde B)>0.2$ for at least $0.2\cdot 25C$ values of $i$.  Each of these $v_i$'s are adjacent to at least $0.2|\tilde B|$ of the vertices in $\tilde B$, so there must be some vertex $w\in \tilde B$ which is adjacent to at least $0.2^2\cdot 25C = C$ of these $v_i$'s.  Renumber the parts so that $\tilde B= B_{2C+1}$ and $w\in \tilde B$ is adjacent to each of $v_1,\dots, v_C$, and we have the desired structure.
\end{proof}

Applying the same argument in the opposite direction, we obtain a symmetric structure.

\begin{claim}
   $G$ contains parts $A_1',\dots, A_C' \in \mathcal R$, $B_1',\dots, B_{2C+1}' \in \mathcal L$, distinct from the $A_i$'s and $B_i$'s, and vertices $v_i'\in A_i'$ and $w' \in B_{2C+1}'$ satisfying the conditions of Claim \ref{claim:bipabsorb}.
\end{claim}

Let $\A_L = \{A_1,\dots, A_C\}\subset \mathcal L$, $\B_R = \{B_1,\dots, B_{2C+1}\} \subset \mathcal R$ and $\A_R = \{A_1',\dots, A_C'\}\subset \mathcal R$, $\B_L = \{B_1',\dots, B_{2C+1}'\} \subset \mathcal L$. At this point we have two connected components: one traversing $\A_L\cup\B_R$ and one traversing $\A_R\cup\B_L$.  Let $\LL' = \LL\setminus(\A_L\cup\B_L)$ and $\R' = \R\setminus(\A_R\cup\B_R)$.  The next step is to partition $\LL'$ and $\R'$ into a small number of connected components, which is accomplished by the next claim.

\begin{claim}\label{claim:stars}
There exists an integer $c_1$ and distinct parts $P_1,\dots, P_{c_1}\in \mathcal L'$ and $Q_1,\dots, Q_{2c_1}\in\mathcal R'$, and vertices $u_i\in P_i$, such that the following hold: 
\begin{enumerate}
\item $d(u_i,Q_{2i-1}) + d(u_i,Q_{2i}) \geq 1$ for each $1\leq i \leq c_1$.
\item Every part in $\mathcal R'$, except for at most one, contains a neighbor of at least one of $u_1,\dots, u_{c_1}$.
\item $c_1 \leq \lceil \log_2 r \rceil$.
\end{enumerate}

\end{claim}

\begin{proof}
We iterate the following process, which is essentially the same as the one used in Claim \ref{claim:bipabsorb}. For each $1\leq i$, let $\mathcal Q_i$ be the set of parts in $\mathcal R'$ which contain no neighbors of any of $u_1,\dots, u_{i-1}$.  On step $i$,
\begin{enumerate}
\item Choose $P_i$ to be any part in $\mathcal L'\setminus \{P_1,\dots, P_{i-1}\}$.
\item By averaging, we may choose a vertex $u_i\in P_i$ such that 

$\sum_{Q\in \mathcal Q_i}d(u_i,Q) > \frac 1 2 |\mathcal Q_i|$.
\item Again by averaging, as long as $|\mathcal Q_i| \geq 2$, we may choose $Q_{2i-1}, Q_{2i}\in\mathcal Q_i$ such that $d(u_i,Q_{2i-1}) + d(u_i,Q_{2i}) > 1$.
\end{enumerate}
We run this process until $\mathcal |\mathcal Q_i|\leq 1$, at which point we can no longer perform step 3.  Let $c_1$ be the largest $i$ for which $|\mathcal Q_i|\geq 2$; then $P_1,\dots, P_{c_1}$, $ Q_1,\dots, Q_{2c_1}$, and $u_1,\dots, u_{c_1}$ satisfy condition (1), and $|\mathcal Q_{i+1}|\leq 1$ implies condition (2).  To bound $c_1$, we claim that for each $i$, $|\mathcal Q_{i+1}| \leq \frac 1 2 |\mathcal Q_{i}|$.  Indeed, since $u_{i}$ was chosen such that $\sum_{Q\in \mathcal Q_{i}}d(u_i,Q) > \frac 1 2 |\mathcal Q_i|$, and each term in the sum is bounded by 1, at most $\frac 1 2 |\mathcal Q_i|$ terms in the sum may be equal to 0; thus, at least $\frac 1 2 |\mathcal Q_i|$ parts in $\mathcal Q_i$ contain a neighbor of $u_i$.  Therefore, $c_1 \leq \lceil \log_2|\mathcal R'|\rceil \leq \lceil \log_2 r\rceil$.

\end{proof}

Denote by $T_R$ the part in $\mathcal R'$ which does not contain any neighbor of $u_1,\dots, u_{c_1}$, if one exists.  Let $\mathcal L'' = \mathcal L'\setminus \{P_1,\dots, P_{c_1}\}$ and $\mathcal R'' = \mathcal R'\setminus \{R_1,\dots, R_{2c_1},T_R\}$.  Again we may apply the same argument in the opposite direction to obtain a symmetric structure.

\begin{claim}
    $G$ contains parts $P_1',\dots, P_{c_2}'\in\mathcal R''$, $Q_1',\dots, Q_{2c_2}'\in\mathcal L''$, and vertices $u_i'\in P_i'$ such that $d(u_i',Q'_{2i-1}) + d(u_i',Q_{2i}') \geq 1$ for each $1\leq i \leq c_2$, at most one part in $\mathcal L''$ does not contain any neighbor of any $u_1',\dots,u_{c_2}'$, and $c_2\leq \lceil \log_2 r \rceil$.
\end{claim}
Again, denote by $T_L$ the part in $\mathcal L''$ which does not contain any neighbor of any $u_1',\dots,u_{c_2}'$, if one exists.

We finish by finding a connected traversal.  Let $c = \max(c_1, c_2)$; we have $c \leq \left\lceil\log_2 r\right\rceil <  C-2$.  By Lemma \ref{lem:connect4strong}, the following paths exist in $G$:

\begin{itemize}
\item For each $1\leq i \leq c_1$, a path using at most one vertex in each of $B_{2i-1}, B_{2i}$, $ Q'_{2i-1}, Q'_{2i}$ with one endpoint in $N_{B_{2i-1}}(v_i)\cup N_{B_{2i}}(v_i)$ and one in $N_{Q_{2i-1}'}(u_i')$ $\cup N_{Q_{2i}'}(u_i')$.
\item For each $1\leq i\leq c_2$, a path using at most one vertex in each of $B_{2i-1}', B_{2i}'$, $Q_{2i-1}, Q_{2i}$ with one endpoint in $N_{B_{2i-1}'}(v_i')\cup N_{B_{2i}'}(v_i')$ and one in $N_{Q_{2i-1}}(u_i)$ $\cup N_{Q_{2i}}(u_i)$.
\item A path using at most one vertex in each of $B_{2c+1},B_{2c+2},B_{2c+1}',B_{2c+2}'$ with one endpoint in $N_{B_{2c+1}}(v_{c+1}) \cup N_{B_{2c+2}}(v_{c+1})$ and one in 

$N_{B_{2c+1}'}(v_{c+1}')\cup N_{B_{2c+2}'}(v_{c+1}')$.
\end{itemize}

And furthermore since $d(T_L,B_{2c+3}), d(T_L, B_{2c+4}), d(T_R,B_{2c+3}')$ and $d(T_R, B_{2c+4}')$ are all greater than $\frac 1 2$, we are guaranteed to find
\begin{itemize}
\item An edge with one endpoint in $T_L$ and one in 

$N_{B_{2c+3}}(v_{c+2})\cup N_{B_{2c+4}}(v_{c+2})$ (if $T_L$ exists).
\item An edge with one endpoint in $T_R$ and one in 

$N_{B_{2c+3}'}(v_{c+2}')\cup N_{B_{2c+4}'}(v_{c+2}')$ (if $T_R$ exists).
\end{itemize}

By choosing vertices $v_i,v_i',u_i,u_i',w,$ and $w'$ from their respective parts, choosing vertices in each $B_i, B_i', Q_i, Q_i', T_L$, and $T_R$ to form the above paths, and choosing vertices adjacent to some $v_i, v_i', u_i,$ or $u_i'$ from all other parts, we obtain a connected traversal.

\end{proof}

\section{$K_t$-factor traversals}\label{sec:factors}

In this section we prove Theorems \ref{thm:ktfactor} and \ref{thm:bipfactor}, and sketch the proof of Theorem \ref{thm:oddcycfactor}. We begin with a general definition and an observation regarding the critical density for a general $F$-factor traversal.

\begin{definition}
Let $F$ be a graph.  An \emph{$F$-subtraversal} is a subtraversal of $G$ which is isomorphic to $F$.  We say that an $r$-partite graph $G$ with vertex partition $\mathcal P = \{P_1,\dots, P_r\}$ is \emph{locally $F$-covered} if for every part $P\in\mathcal P$, $G$ contains an $F$-subtraversal which contains some vertex in $P$.
\end{definition}

We observe that if an $r$-partite graph $G$ is not locally $F$-covered, it cannot contain an $F$-factor traversal.  The next lemma states that this sort of local obstruction is, in an asymptotic sense, the only barrier to finding an $F$-factor traversal.

\begin{lemma}\label{lem:factorabsorb}
Let $F$ be a graph on $t$ vertices, and let $\alpha \in [0,1]$.  Suppose that there is a constant $r_F$ such that, for all $r > r_F$, any $r$-partite graph $G$ with $d_r(G) > \alpha$ is locally $F$-covered.   Then there exists $r_0$ such that $\pi_{rt}(rF) \leq \alpha$ for all $r > r_0$.
\end{lemma}

\begin{proof}
Let $G$ be an $rt$-partite graph with vertex partition $\mathcal P = \{P_1,\dots, $ $P_{rt}\}$ with $d_r(G) > \alpha$.  Let $C = tr_F$.  Let $r_0 = R_t(C, r_F)$ be the $t$-uniform hypergraph Ramsey number.  We show that if $r > r_0$ then $G$ contains an $F$-factor traversal.  

The first step is to find an absorber, which consists of a set $\mathcal A\subset\mathcal P$ of size $C$ such that any $t$ parts of $\mathcal A$ contain an $F$-subtraversal.  Consider the $t$-uniform hypergraph $\mathcal H$ whose vertex set is $\mathcal P$ and which has an edge $\{P_1,\dots, P_t\}$ if $\{P_1,\dots, P_t\}$ contains an $F$-subtraversal.  By assumption, every subset of $\mathcal P$ of size $r_F$ is locally $F$-covered and in particular contains an $F$-subtraversal, so $\mathcal H$ has no independent set of size $r_F$.  Thus as long as $r \geq r_0$, $\mathcal H$ contains a complete subhypergraph of size $C$, whose vertices form the elements of $\mathcal A$.

Let $\mathcal P' = \mathcal P\setminus \mathcal A$.  As observed previously, every subset of $\mathcal P'$ of size at least $r_F$ contains an $F$-subtraversal, so by iterating we may find a collection of vertex-disjoint copies of $F$ which traverses all but at most $r_F$ of the parts in $\mathcal P'$.  Let $\mathcal P''$ be the set of parts which are not traversed.  

Since $|\mathcal A\cup \mathcal P''| = tr_F + |\mathcal P''| \geq r_F + t|\mathcal P''|$, we have that $\mathcal A\cup\mathcal P''$ is locally $F$-covered, so there is some $F$-subtraversal which contains a vertex in at least one part of $\mathcal P''$ (and hence at most $t-1$ parts in $\mathcal A$).  Iterating this argument, there exists a set $S\subset\mathcal A$ of size at most $(t-1)|\mathcal P''|$ such that $\mathcal S\cup\mathcal P''$ contains an $F$-factor traversal.  Finally, by the construction of $\mathcal A$, we have that $\mathcal A\setminus \mathcal S$ must contain an $F$-factor traversal; this completes the proof.

\end{proof}

This lemma enables simple proofs of Theorems \ref{thm:ktfactor}, \ref{thm:bipfactor}, and \ref{thm:oddcycfactor}.  For the first proof we need the traversal version of the Erd\H os-Stone-Simonovitz theorem established by Bondy, Shen, Thomass\'e, and 
Thomassen.

\begin{lemma}[\cite{bondy}, Theorem 1]\label{lem:ess}
    Let $F$ be a graph and $\chi(F)$ its chromatic number.  Then $\pi_r(F) \geq 1 - \frac{1}{\chi(F)-1}$ for all $r$, and
    $\lim_{r\to\infty}\pi_r(F)  =1 - \frac{1}{\chi(F)-1}.$
\end{lemma}

The proof is an averaging argument -- any $r$-partite graph $G$ with $d_r(G) >\alpha$ contains a traversal with edge density at least $\alpha$, at which point we can directly apply the normal Erd\H os-Stone-Simonovitz theorem.  The same sort of argument gives a multipartite analogue of the K\H ov\' ari-S\'os-Tur\' an theorem, which we need for the second proof.

\begin{lemma}\label{lem:kst}
    For any bipartite graph $F$, $\lim_{r\to\infty}\pi_{K_{r,r}}(F) = 0$.
\end{lemma}

\subsection{Proof of Theorem \ref{thm:ktfactor}}

\subsubsection{The lower bound}

Recall that $\alpha_t$ is the positive solution to the quadratic $\alpha = 1-\frac 1 {t-2}\alpha^2$.  The following construction illustrates the lower bound $\pi_{rt}(rK_t) \geq \alpha_t$.

\begin{construction}
Let $t \geq 3$, $r \geq 1$, and let $\alpha \in [0,1]$ be a parameter.  We construct a weighted $rt$-partite graph $G$ as follows.  By Lemma \ref{lem:ess}, there exists an $(rt-1)$-partite graph $G'$ with parts $V_1',\dots, V_{rt-1}'$ and density $d_{rt-1}(G') \geq 1-\frac{1}{\chi(K_{t-1})-1} = \frac{t-3}{t-2}$ with no $K_{t-1}$ traversal.  For each $1\leq i<rt$, the part $V_i$ of $G$ is obtained by taking $V_{i}'$ along with an extra vertex $u_i$; the weights are scaled so that $w(u_i) = 1-\alpha$ and the total weight of all vertices in $V_i'$ is $\alpha$.  The final part $V_{rt}$ consists of a single vertex $\{v\}$ with $w(v) = 1$.

For each $1\leq i<rt$, there is an edge between $v$ and every vertex in $V_i'$.  For $1\leq i \neq j < rt$, there are edges between $u_i$ and every vertex in $V_j$.  The edges between $V_i'$ and $V_j'$ are those that were present in $G'$.

Any traversal of $G$ must contain the vertex $v$, but the neighborhood of $v$ in $G$ is only those vertices in $G'$, which by construction contains no $K_{t-1}$.  Hence $G$ is not locally $K_t$-covered and has no $K_t$-factor traversal.  The edge density between parts $V_i$ and $V_j$ is $\alpha$ if $i$ or $j$ equals 1, and is at least $1-\alpha^2 + d_{rt-1}(G') \alpha^2 \geq 1-\frac{1}{t-2}\alpha^2$ otherwise.  Choosing $\alpha$ to satisfy $\alpha = 1-\frac{1}{t-2}\alpha^2$ yields $d_{rt}(G) \geq \alpha_t$.

\end{construction}

\subsubsection{The upper bound}
\begin{proof}[\unskip\nopunct]
For the upper bound, we prove that for all $\ep > 0$, there is an $r_{K_t}(\ep)$ such that, for every $r > r_{K_t}(\ep)$, any $r$-partite graph $G$ with $d_{r}(G) > \alpha_t + \ep$ is locally $K_t$-covered.  By Lemma \ref{lem:factorabsorb}, this implies $\lim_{r\to\infty}\pi_{rt}(rK_t) \leq \alpha_t$, completing the proof.

Fix $\ep > 0$.  By Lemma \ref{lem:ess}, there exists a constant $r_{t,\ep}$ such that $\pi_{r_{t,\ep}}(K_{t-1}) < \frac{t-3}{t-2} + \ep$.  Let $G$ be an $r$-partite graph for $r$ sufficiently large with vertex partition $\{ P_1, \dots,  P_r\}$ and $d_r(G) > \alpha_t + \ep$, and let $P$ be any part in $G$. By Lemma \ref{lem:robustconnect}, there is a vertex $v\in P$ and parts $P_1,\dots, P_{r_{t,\ep}}$ such that $d(v, P_i) > \alpha_t$ for each $1\leq i\leq r_{t,\ep}$.  For each $1\leq i < j\leq r_{t,\ep}$, we have
\begin{align*}\alpha_t + \ep &< d(P_i, P_j)\\
&\leq d(v, P_i)d(v,P_j)d(N_{P_i}(v), N_{P_j}(v)) + 1-d(v, P_i)d(v,P_j)\\
&\leq 1-\alpha_t^2 \left(1-d(N_{P_i}(v), N_{P_j}(v))\right).
\end{align*}
So $d(N_{P_i}(v), N_{P_j}(v))\geq \frac{\alpha_t + \alpha_t^2 - 1 + \ep}{\alpha_t^2} \geq \frac{t-3}{t-2} +\ep.$  It follows from our choice of $r_{t,\ep}$ that the vertex sets $N_{P_1}(v), \dots, N_{P_{r_{t,\ep}}}(v)$ contain a $K_{t-1}$-subtraversal; by adding $v$ we obtain a $K_t$-subtraversal using a vertex in $P$.

\end{proof}

\subsection{Proof of Theorem \ref{thm:bipfactor}}

Let $F$ be a connected bipartite graph on $t$ vertices which is not a star.  Let $\alpha = \frac{3-\sqrt 5}{2}$.

\subsubsection{The lower bound}
The following construction illustrates the lower bound $\pi_{rt}(rF) \geq \alpha$ for any $t$-vertex graph $F$ which is not a star.

\begin{construction}\label{cons:pathlb}
    For $r\geq 1$, we define a weighted $rt$-partite graph $G$.  The parts of $G$ are $V_i = \{x_i,y_i\}$ for $i = 1, 2, \dots, rt-1$, and $V_{rt} = \{v\}$.  There are edges $x_iv$ for all $1\leq i < rt$, and $y_iy_j$ for all $1\leq i < j <rt$.  The weights are $w(x_i) = \alpha$, $w(y_i) = 1-\alpha$ for all $1\leq i < rt$, and $w(v) = 1$.  $G$ is not locally $F$-covered, since the connected component of $G$ containing $v$ is only a star so no there is no $F$-subtraversal containing $v$. And $d_{rt}(G) = \min(\alpha, (1-\alpha)^2) = \alpha$.
\end{construction}

\subsubsection{The upper bound}
\begin{proof}[\unskip\nopunct]

For the upper bound, we again prove that for all $\ep > 0$, there is an $r_{F}(\ep)$ such that, for every $r > r_{F}(\ep)$, any $r$-partite graph $G$ with $d_{r}(G) > \alpha + \ep$ is locally $F$-covered.  By Lemma \ref{lem:factorabsorb}, this implies $\lim_{r\to\infty}\pi_{rt}(rF) \leq \alpha$, completing the proof.

Fix $\ep > 0$.  Let $G$ be an $r$-partite graph for $r$ sufficiently large with vertex partition $\{ P_1, \dots,  P_r\}$ and $d_r(G) > \alpha + \ep$, and let $P$ be any part in $G$. We find a $K_{t,t+1}$-subtraversal containing a vertex in $P$; this necessarily contains an $F$-subtraversal containing a vertex in $P$.

 By Lemma \ref{lem:kst} there exists a constant $r_\ep$ such that $\pi_{K_{r_\ep,r_\ep}}(K_{t,t}) < \ep/2$.  Let $C$ be large enough that any balanced bipartite graph on $2C$ vertices with at least $\frac{C(C-1)}{2}$ edges contains a copy of $K_{r_\ep,r_\ep}$.  By Lemma \ref{lem:robustconnect}, there is a vertex $v\in P$ and parts $P_1,\dots, P_{C}\in \mathcal P$ such that $d(v, P_i) > \alpha + \frac \ep 2$ for each $1\leq i\leq C$.   For each $1\leq i < j\leq C$, we have
\begin{align*}\alpha + \ep &< d(P_i, P_j)\\
&\leq d(N_{P_i}(v),P_j) + d(P_i, N_{P_j}(v)) + d(P_i\setminus N_{P_i}(v), P_j\setminus N_{P_j}(v))\\
&\leq d(N_{P_i}(v),P_j) + d(P_i, N_{P_j}(v)) + \left(1-\alpha - \frac \ep 2\right)^2.
\end{align*}
So
\[d(N_{P_i}(v),P_j) + d(P_i, N_{P_j}(v)) \geq \alpha + \ep - (1-\alpha)^2 + \ep(1-\alpha) - \frac {\ep^2}{4} > \ep,\]
and hence at least one of $d(N_{P_i}(v),P_j)$ and $d(P_i, N_{P_j}(v))$ is at least $\frac \ep 2$.

Consider an auxiliary bipartite graph $H_P$ whose left and right vertex sets $\mathcal L$ and $\mathcal R$ are two copies of $\{P_1,\dots, P_C\}$ with an edge from $P_i\in \mathcal L$ to $P_j \in \mathcal R$ if $d(N_{P_i}(v),P_j) \geq \frac \ep 2$.  The above shows that for every $1\leq i < j \leq C$, there is either an edge from $P_i\in\mathcal L$ to $ P_j\in\mathcal R$ or $P_j\in\mathcal L$ to $ P_i\in\mathcal R$; in particular the number of edges in $H_P$ is at least $\frac{C(C-1)}{2}$.  By our choice of $C$, $H_P$ has a $K_{r_\ep, r_\ep}$ subgraph; that is, there are distinct parts $P_1,\dots, P_{r_\ep}$ and $Q_1,\dots, Q_{r_{\ep}}$ in $\{P_1,\dots, P_C\}$ such that $d(N_{P_i}(v),Q_j) \geq \frac \ep 2$ for all $1\leq i, j \leq r_\ep$. 

Consider now the $K_{r_\ep,r_\ep}$-partite graph $G'$ whose vertex partition is $\mathcal L = \{N_{P_1}(v),\dots,N_{P_{r_\ep}}(v)\}$ and $\mathcal R = \{Q_1,\dots, Q_{r_\ep}\}$.  Then $d_{K_{r_\ep,r_\ep}}(G') \geq \frac \ep 2$, so by our choice of $r_\ep$ it contains a $K_{t,t}$-subtraversal.  This gives a $K_{t,t}$-subtraversal of $G$ of which half the vertices lie in the neighborhood of $v$; by including $v$ we obtain a $K_{t,t+1}$-subtraversal using a vertex in $P$, as desired.
\end{proof}

\subsection{Sketch of proof of Theorem \ref{thm:oddcycfactor}}

As the proof is very similar in structure to the previous two, and very similar in methods to the proof of Theorem \ref{thm:bipham}, we omit details.  The lower bound is given by Construction 4.2 in \cite{badakhshian}, which for any $r\geq 1$ gives an example of an $r$-partite graph with $d_r(G) = \frac 1 2$ and no odd cycle subtraversal.

For the upper bound, we again argue that for all odd $t\geq 7$ and $\ep > 0$, any $r$-partite graph $G$ with $d_r(G) > \frac 1 2 + \ep$ and $r$ sufficiently large is locally $C_t$-covered.  Let $t' = \frac{t-1}{2}$.  For any part $P$ of $G$, applying Lemma \ref{lem:robustconnect} to $P$, and then the argument of Claim \ref{claim:absorber} to the parts which $P$ robustly connects, yields parts $P_1,\dots, P_{t'}$ and $Q_1,\dots, Q_{t'+3}$ such that $P_1 = P$ and each $P_i$ robustly connects $Q_1,\dots, Q_{t'+3}$.  By applying Lemma \ref{lem:connect4strong} we can obtain a subtraversal path from the connecting vertex in $P_1$ to that in $P_2$ which traverses either 2 or 4 parts in $Q_1,\dots, Q_4$.  Then among the parts $P_1,\dots, P_{t'}$ and $Q_5,\dots, Q_{t'+3}$ we can construct a robust sequence  with length either $t-3$ or $t-5$ whose end parts are $P_1$ and $P_2$.  This completes a $C_t$-subtraversal using some vertex in $P$.

\section{Proof of Lemma \ref{lem:connect4strong}}\label{sec:lemmaproof}

In this section we prove Lemma \ref{lem:connect4strong}.

\begin{proof}
    Suppose that neither of the given conditions occur. Define the following sets:
    \begin{align*}A = N_{P_1}(S_3)\setminus N_{P_1}(S_4),\; &B = N_{P_1}(S_3)\cap N_{P_1}(S_4), \\\;& C = N_{P_1}(S_4)\setminus N_{P_1}(S_3) \subset P_1\end{align*}
    \begin{align*}
    D = N_{P_2}(S_3)\setminus N_{P_2}(S_4),\; &E = N_{P_2}(S_3)\cap N_{P_2}(S_4), \;\\& F = N_{P_2}(S_4)\setminus N_{P_2}(S_3) \subset P_2
    \end{align*}
    \begin{align*}G = N_{P_3}(S_1)\setminus N_{P_3}(S_2),\; &H = N_{P_3}(S_1)\cap N_{P_3}(S_2), \; \\&I = N_{P_3}(S_2)\setminus N_{P_3}(S_1)\subset P_3\end{align*}    \begin{align*}J = N_{P_4}(S_1)\setminus N_{P_4}(S_2),\; &K = N_{P_4}(S_1)\cap N_{P_4}(S_2), \; \\&L = N_{P_4}(S_2)\setminus N_{P_4}(S_1)\subset P_4\end{align*}
     Since there is no edge between $S_1$ and $S_3$, the sets $S_1$ and $N_{P_1}(S_3)$ are disjoint; likewise $S_1$ and $N_{P_1}(S_4)$ are disjoint, so the sets $S_1, A, B$, and $C$ are disjoint subsets of $P_1$.  Likewise $S_2, D, E$, and $F$ are disjoint in $P_2$, and similarly for $P_3$ and $P_4$.

     Let $i, j$ be $1, 2$ in some order, and $k, l$ be $3, 4$ in some order. If there were any edge between $N_{P_i}(S_k)$ and $N_{P_l}(S_j)$, then there would be a path $S_k, P_i, P_l, S_j$, which by assumption does not exist.  It follows that the following pairs of sets may not have any edges between them:
     \[N_{P_1}(S_3) = A\cup B \text{ and } N_{P_4}(S_2) = K\cup L\]
     \[N_{P_1}(S_4) = B\cup C \text{ and } N_{P_3}(S_2) = H\cup I\]
     \[N_{P_2}(S_3) = D\cup E \text{ and }N_{P_4}(S_1) = J\cup K\]
     \[N_{P_2}(S_4) = E\cup F \text{ and }N_{P_3}(S_1) = G\cup H\]
     Between $P_1$ and $P_3$, then, the maximum number of edges is 
     \[|S_1||N_{P_3}(S_1)| + |N_{P_1}(S_3)||S_3| + (|P_1|-|S_1|)(|P_3|-|S_3|) - |N_{P_1}(S_4)||N_{P_3}(S_2)|.\]
     Denoting $\frac{|A|}{|P_1|}$ by $a$ and defining $b, c, \dots, l$ and $s_1,\dots, s_4$ similarly, the maximum value of $d(P_1,P_3)$ is
     \[s_1(g + h) + s_3(a + b) + \left(1-s_1\right)\left(1-s_3\right) - (b + c)(h + i).\]
     And similarly,
\begin{align*}&d(P_1,P_4)\leq s_1(j+k)+s_4(b+c)+(1-s_1)(1-s_4)-(a+b)(k+l),\\
&d(P_2,P_3)\leq s_2(h+i)+s_3(d+e)+(1-s_2)(1-s_3)-(e+f)(g+h),\\
&d(P_2,P_4)\leq s_2(k+l)+s_4(e+f)+(1-s_2)(1-s_4)-(d+e)(j+k).\end{align*}

    Denote these four expressions by $\Delta_1$, $\Delta_2$, $\Delta_3$, $\Delta_4$.

     \begin{claim*} Subject to the constraints that all variables are nonnegative, $s_1+s_2\geq1$, $s_3+s_4\geq1$, and 
     \[ s_1+a + b + c, \; s_2+d + e + f, \;s_3+g + h + i, \; s_4+j + k + l \leq  1, \]
     we have that $\Delta := \Delta_1+\Delta_2+\Delta_3+\Delta_4$ is at most $2$.
     \end{claim*}

     Given this claim, at least one of the four densities must be at most $\frac 1 2$, contradicting the density condition on $G$.
     \begin{proof}[Proof of claim]\let\qed\relax
         Let $s_1,\dots,s_4,a, b, \dots, l$ be values maximizing $\Delta$ subject to the given constraints.

	We first claim that without loss of generality we many assume $s_1+a+b+c=1$.  Indeed, the terms of $\Delta$ involving $b$ are $b(s_3-h-i+s_4-k-l)$.  Since 
    \[(s_3-h-i+s_4-k-l) + (s_3 + h + i) + (s_4 + k + l) = 2(s_3 + s_4) \geq 2,\]
    but the latter two summands are each at most one, it must be that $s_3-h-i+s_4-k-l$ is nonnegative.  Hence if $s_1+a+b+c<1$, we may increase $b$ until $s_1+a+b+c=1$ without decreasing $\Delta$.  Similarly, we may assume
         \[s_2+ d + e + f, \;s_3+g + h + i, \; s_4+j + k + l =1.\]
         Substituting $b = 1-s_1 - a - c$, $e = 1-s_2 - d - f$, $h = 1-s_3 - g - i$, $k=1-s_4- j - l$, we get
         \begin{align*}
             \Delta_1  =\; &s_1(1-s_3-i) + s_3(1-s_1-c) + \left(1-s_1\right)\left(1-s_3\right) \\&\;- (1-s_1-a)(1-s_3-g)\\
             =\;& s_1+s_3-2s_1s_3  - is_1 - cs_3+ (1-s_3)a+(1-s_1)g-ag.\\
            \end{align*}
        And similarly,
        \begin{align*}&\Delta_2 = s_1+s_4-2s_1s_4-ls_1-as_4+(1-s_4)c+(1-s_1)j-cj,\\
        &\Delta_3 =s_2+s_3-2s_2s_3-gs_2-fs_3+(1-s_3)d+(1-s_2)i-di,\\
        &\Delta_4 = s_2+s_4-2s_2s_4-js_2-ds_4+(1-s_4)f+(1-s_2)l-fl.\end{align*}
        Adding these together and factoring yields
        \begin{align*}
\Delta =\;& 2\left(s_1+s_2+s_3+s_4-(s_1+s_2)(s_3+s_4)\right)\\&-(s_3+s_4-1)(a+c+d+f)-(s_1+s_2-1)(g+i+j+l)\\&-ag-cj-di-fl\\
\leq\;& 2\left(s_1+s_2+s_3+s_4-(s_1+s_2)(s_3+s_4)\right)\\
=\;& 2\left(1-(s_1+s_2-1)(s_3+s_4-1)\right)\\
\leq\;&2(1-0) = 2.
\end{align*}
     \end{proof}

\end{proof}

\section{Discussion and open problems}\label{sec:conclusion}

In this section we discuss our results and remark on some remaining open problems.

In Section \ref{sec:bipcon} we proved $\pi_{K_{r,r}}(\mathcal T_{2r}) = \frac 1 2$ for $r > 7000$.  We strongly conjecture that this equality holds for all $r$. 
\begin{conjecture}
    $\pi_{K_{r,r}}(\mathcal T_{2r}) = \frac 1 2$ for all $r \geq 1$.
\end{conjecture}

Indeed by simple arguments one can show $\pi_{K_{2,2}}(\mathcal T_4) = \frac 1 2$ (this follows from Theorem 1.6 of \cite{badakhshian}, as $K_{2,2}$ is $K_4$ with a matching removed) and $\pi_{K_{3,3}}(\mathcal T_6) = \frac 1 2$ (if $G$ is a $K_{3,3}$-partite graph, one may find vertices $l \in L_1$, $r\in R_1$ such that $d(l,R_2) + d(l,R_3), d(r,L_2) + d(r,L_3) > 1$, and then apply Lemma \ref{lem:connect4strong} to find a path from $N_{R_2\cup R_3}(l)$ to $N_{L_2\cup L_3}(r)$).

In Section \ref{sec:bipham} we proved $\lim_{r\to\infty}\pi_r(C_r) = \frac 1 2$.  A particularly natural next question is to ask for the critical density for a Hamiltonian path: 

\begin{question}
    What are the values of $\pi_r(P_r)$ and $\pi_{K_{r,r}}(P_{2r})$?
\end{question}

Clearly $\pi_r(P_r) \leq \pi_r(C_r)$.  We can also observe that, since $P_r$ is not a star for $r\geq 4$, Construction \ref{cons:pathlb} gives the lower bound $\pi_r(P_r)\geq \frac{3-\sqrt{5}}{2}$ for $r\geq 4$.  By exhaustive casework the values of $\pi_r(P_r)$ may be computed for small $r$ (see e.g. \cite{badakhshian} \S 2.1 for discussion, and Proposition A.1 for an example of this method); such exhaustive search confirms $\pi_4(P_4) = \frac{3-\sqrt 5}{2}$ exactly, but appears untenable for larger $r$. We tentatively conjecture that $\pi_r(P_r) = \frac{3-\sqrt{5}}{2}$ for $r\geq 4$, though it would be interesting to show even that $\pi_r(P_r) \leq \frac 1 2$; such a statement seems out of reach of the methods of this paper.  

In the bipartite setting, the fact that $\pi_{K_{r,r}}(\T_{2r})\leq\pi_{K_{r,r}}(P_{2r})\leq\pi_{K_{r,r}}(C_{2r})$ implies that $\lim_{r\to\infty}\pi_{K_{r,r}}(P_{2r}) = \frac 1 2$.  The above arguments showing that $\pi_{K_{2,2}}(\T_4) = \pi_{K_{3,3}}(\T_6) = \frac 1 2$ also show that $\pi_{K_{2,2}}(P_4) = \pi_{K_{3,3}}(P_6) = \frac 1 2$, since in both cases the spanning tree produced is a path.  We conjecture that $\pi_{K_{r,r}}(P_{2r}) = \frac 1 2$ for all $r\geq 1$.

In Section \ref{sec:factors} we examined factor traversals.  Our method was general and allowed us to determine asymptotically the critical density for an $F$-factor traversal when $F$ is complete, bipartite, or an odd cycle of length at least 7.  Naturally we may ask for the same value for other graphs $F$:

\begin{question}
    What is the value of $\lim_{r\to\infty}\pi_{r|V(F)|}(rF)$ for an arbitrary graph $F$?
\end{question}

In particular we note that $F = C_5$ is the only cycle for which we don't know this value. It seems likely that the method of Section \ref{sec:factors} may generalize to answer this question for a wider variety of graphs $F$. 

We may also ask when the critical density for an $F$-factor traversal may be determined exactly.  As remarked in the introduction, since $\pi_3(K_3) = \frac{-1+\sqrt 5}{2} = \alpha_3$, we have the exact result $\pi_{3r}(rK_3) = \pi_3(K_3) = \alpha_3$ for all $r\geq 1$.  As $\pi_4(P_4) = \frac{3-\sqrt 5}{2}$, and $P_4$ is bipartite, this likewise gives the exact result $\pi_{4r}(rP_4) = \pi_4(P_4) = \frac{3-\sqrt 5}{2}$.  We do not have a similar result for larger complete graphs; while the value of $\pi_t(K_t)$ is not known precisely for $t \geq 4$, a construction of Csikv\'ari and Nagy \cite{csikvari-nagy} yields $\pi_t(K_t)\geq \beta_t > \alpha_t$, where $\beta_t$ is defined recursively by 
\[\beta_t = \begin{cases}0 & t = 2 \\ \text{the positive solution to }\beta = 1-(1-\beta_{t-1})\beta^2 & t \geq 3\end{cases}.\]
We therefore ask whether or not the asymptotic answer might nonetheless be exact for larger $t$.
\begin{question}
    For $t\geq 4$, does $\pi_{rt}(rK_t) = \alpha_t$ for any $r > 0$?
\end{question}

And in general it would be interesting to characterize the behavior of the sequence $(\pi_{r|V(F)|}(rF))_{r=1}^\infty$, similar to the characterization of the behavior of the sequence $(\pi_r(F))_{r=1}^\infty$ for an arbitrary graph $F$ given by Narins and Tran in \cite{narins}.

\begin{question}
    For which $t$-vertex graphs $F$ do we have $\pi_{rt}(rF) = \pi_{t}(F)$ for all $r$? Among graphs for which $\pi_{rt}(rF) < \pi_{t}(F)$ for some $r$, which, if any, have the sequence $\pi_{rt}(rF)$ eventually constant?
\end{question}

We note that Lemma \ref{lem:factorabsorb} may still be useful to answer these questions, if one could establish that sufficiently large $r$-partite graphs are locally $F$-covered without relying on Lemma \ref{lem:robustconnect}.

As a final remark, we note that $r$-partite density is not necessarily the only metric of interest in this context.  Minimum degree conditions are also sensible to study: multipartite versions of the Dirac and Hajnal-Szemeredi theorems involve conditions on the $r$-partite, or ``local", minimum degree
\[\delta_r(G) := \min_{\substack{1\leq i \neq j \leq r\\ v\in V_i}}|N_{V_j}(v)|.\]
Furthermore there are several classes of $r$-partite graph traversals, including independent/complete traversals and traversals with each connected component a bounded size, which are widely studied and for which many known sufficient conditions for their existence involve (global or local) minimum or maximum degree \cite{berke, haxell, loh}.

It is interesting to consider the minimum degree variants of the problems studied in this paper.  For many of them the answers are relatively simple: it is easy to see, for example, that in an $r$-partite graph $G$ with each part of size $n$, $\delta_r(G) > \left(1 - \frac{1}{t-1}\right)n$ is both necessary and sufficient to guarantee the existence of a $K_t$-factor-traversal.  (For the lower bound consider an $r$-partite graph on parts $V_i = \{v_{i,1},\dots, v_{i,{t-1}}\}$ with edges $v_{i,k}v_{j,k'}$ if and only if $k\neq k'$.) For Hamiltonian traversals, however, the situation is somewhat less clear.  It is easy to see $\delta_r(G) > \frac 1 2 n$ is necessary and sufficient when $r$ is odd (the lower bound construction is the same as that just described for a $K_3$-factor). For even $r$ however, it turns out that $\delta_r(G) \geq \frac 1 2n$ is sufficient (essentially because in any $r$-partite graph with $\delta_r(G) \geq \frac 1 2$ and no $C_r$-traversal, the edges between each pair of parts must form two copies of $K_{n/2,n/2}$, reducing the problem to the case where each part has two vertices), and we do not know a nontrivial lower bound.

\begin{question}
    For even $r$, what is the minimum $\delta(n)$ such that any $r$-partite graph with each part of size $n$ and $\delta_r(G) > \delta(n)$ contains a Hamiltonian traversal?
\end{question}

\subsection*{Acknowledgments}
The author would like to thank Tom Bohman for many helpful discussions and guidance in the preparation of this paper.  This research was partially supported by NSF Award DMS-2246907.

\printbibliography

\end{document}